\documentclass[a4paper,12pt]{scrartcl}

\usepackage{scrpage2}
\usepackage{wrapfig}
\usepackage{graphicx}
\usepackage{soul}

\usepackage[utf8]{inputenc}
\usepackage[T1]{fontenc}

\usepackage{multicol}
\usepackage{needspace}
\usepackage{comment}

\usepackage[normalem]{ulem}

\usepackage{amsmath}
\usepackage{mathtools}
\usepackage{tensor}
\usepackage{leftidx}
\usepackage{amsthm}
\usepackage{enumitem}
\setlist[description]{%
  font={\rmfamily\mdseries \dashuline}, 
}
\usepackage{scalerel}	
\usepackage[varumlaut]{yfonts}
\usepackage{amssymb}
\usepackage{amsfonts}
\usepackage{bbm}
\usepackage{extarrows}
\usepackage{stmaryrd}

\usepackage[mathscr]{euscript}

\usepackage{pgf}
\usepackage{tikz}

\usepackage{enumitem}

\usepackage{fancybox}
\usetikzlibrary{arrows,shapes,trees,matrix,calc}
\usetikzlibrary{decorations.markings}
\usetikzlibrary{matrix}

\usepackage{hyperref}

\usepackage[shortcuts]{extdash}

\newcommand{\kk}{\Bbbk} 

\newcommand{\II}{\mathbb{I}}

\newcommand{\lto}{\longrightarrow}
\newcommand{\lmapsto}{\longmapsto}
\newcommand{\isoto}{\xrightarrow{\sim}}
\newcommand{\lisoto}{\xlongrightarrow{\sim}}

\newcommand{\eps}{\varepsilon}

\newcommand{\id}[1]{\operatorname{id}_{#1}}

\newcommand{\Hom}[2]{ \operatorname{Hom}_{#1}\! \left( #2 \right) }
\newcommand{\intHom}[1]{ \operatorname{\underline{Hom}}\! \left( #1 \right) }
\newcommand{\intHomwoarg}{ \operatorname{\underline{Hom}} }

\newcommand{\Nat}[1]{ \operatorname{Nat}\! \left( #1 \right) }

\newcommand{\ev}{\operatorname{ev}}

\newcommand{\coev}{\operatorname{coev}}
\newcommand{\lcoev}{\widetilde{\operatorname{coev}}}

\newcommand{\leftdual}[1]{\prescript{\vee \! }{}{#1}}

\newcommand{\lact}{\triangleright}
\newcommand{\ract}{\triangleleft}

\newcommand{\cattr}{\operatorname{tr}}

\newcommand{\twcenter}{\operatorname{\widetilde{\mathcal{Z}}}}

\newcommand{\tild}[1]{{\widetilde{#1}}}

\newcommand{\mopp}{{\otimes\! \operatorname{op}}}
\newcommand{\opp}{{\operatorname{op}}} 
\newcommand{\coopp}{{\operatorname{cop}}}
\newcommand{\opcop}{{\opp\coopp}}

\newcommand{\cat}[1]{\mathcal{#1}}

\newcommand{\A}{\cat{A}}

\newcommand{\C}{\cat{C}}

\newcommand{\M}{\cat{M}}
\newcommand{\N}{\cat{N}}

\newcommand{\funre}[1]{\operatorname{Fun_{#1}^{r.e.}}}
\newcommand{\funbalre}[1]{\operatorname{Fun_{#1}^{bal, r.e.}}}

\newcommand{\lmodin}[2]{{#2}\mathsf{\operatorname{{--mod}}}_{#1}}
\newcommand{\lmod}[1]{{#1}\!\operatorname{--mod}}
\newcommand{\rmodin}[2]{\operatorname{mod}_{#1}\text{--}{#2}}
\newcommand{\lcomod}[1]{\prescript{#1}{}{\mathscr{M}}}

\newcommand{\bicomod}[2]{\prescript{#1}{}{\mathscr{M}}^{#2}}

\newcommand{\genHopf}[4]{\prescript{#1}{#3}{\mathscr{M}}^{#2}_{#4}}

\newcommand{\modmon}[2]{{#1}^{#2}}

\newtheoremstyle{indented}{3pt}{3pt}{\addtolength{\leftskip}{5.5em}}{}{\bfseries}{.}{.5em}{}

\theoremstyle{plain}

\newcounter{dummy} 

\newtheorem{theorem}[dummy]{Theorem}
\newtheorem{proposition}[dummy]{Proposition}
\newtheorem{lemma}[dummy]{Lemma}

\theoremstyle{definition}

\newtheorem{definition}[dummy]{Definition}

\theoremstyle{remark}

\newtheorem{remark}[dummy]{Remark}

\theoremstyle{indented}

\usepackage[top=2.0cm, bottom=2.0cm, left=2.0cm, right=2.0cm]{geometry}
\addtokomafont{sectioning}{\rmfamily}

\makeatletter
\renewcommand\maketitle{
   \begin{center}
     {\LARGE\bfseries\@title\par\vspace{0.7em}}
     {\@date}
   \end{center}
}
\makeatother
\title{Category-valued traces for bimodule categories: a representation-theoretic realization}

\begin{document}

\begin{flushright}
\textsf{ZMP-HH/18-5}\\
\textsf{Hamburger Beiträge zur Mathematik Nr. 724}\\
\textsf{Februar 2018}
\end{flushright}

\vspace*{0.5cm}

\begin{center}
\LARGE \textbf{Category-valued traces for bimodule categories: a representation-theoretic realization} \normalsize \\[4ex] Vincent Koppen
\end{center}

\begin{center}
\emph{Fachbereich Mathematik,   Universität Hamburg}\\
\emph{Bereich Algebra und Zahlentheorie}\\
\emph{Bundesstra{\ss}e 55,   D -- 20146  Hamburg}\\
Email: \ \texttt{vincent.koppen@uni-hamburg.de}
\end{center}

\begin{abstract}
\noindent \textbf{Abstract.} 
The category-valued trace assigns to a bimodule category over a linear monoidal category a linear category.
It generalizes Drinfeld centers of monoidal categories and the relative Deligne product of bimodule categories.
In this article, we study bimodule categories that are given as categories of bicomodules over a Hopf algebra.
Our main result is a representation-theoretic realization of the category-valued trace as a category of generalized Hopf bimodules.
\end{abstract}

\section{Introduction and outline}

The category-valued trace assigns to a finite linear bimodule category a finite linear category \cite{fuchs+schaumann+schweigert}.
It unifies several known constructions and generalizes the relative Deligne product as well as the Drinfeld center, see Definition \ref{def:twisted-center} below.
More precisely, for the special case of a tensor category, seen as a bimodule category over itself, the category-valued trace is a variant of the Drinfeld center, for which the right-action functor (i.e.\ tensoring from the right) is \emph{twisted} by the double-right-dual functor.

If the underlying tensor category is concretely given as the tensor category of finite-dimen\-sional modules (or, equivalently, comodules) over a finite-dimensional Hopf algebra $H$, then its Drinfeld center is canonically equivalent to the category of Yetter-Drinfeld modules over $H$.
This category is not only equivalent to the category of modules over the Drinfeld double $D(H)$ of $H$, but also has a description as Hopf bimodules over $H$ \cite{schauenburg, bespalov+drabant}.
This raises the question regarding the relation between more general category-valued traces and categories of Hopf bimodules.
In the present article we obtain such a result.
Without loss of generality, we assume that the bimodule category is given by a bicomodule algebra over a Hopf algebra.
We then show in Theorem \ref{thm:trace-of-bimodule-category-given-by-bicomodule-algebra} that its category-valued trace is canonically equivalent to the category of generalized Hopf bimodules over this bicomodule algebra.

We briefly outline the structure of the article.
In Section \ref{sec:notation-and-background} we provide pertinent background.
In particular we recall from \cite{fuchs+schaumann+schweigert} the definition of the twisted center and the balanced functor that turns it into a category-valued trace.
We also recall in Theorem \ref{thm:module-category-via-algebra} that finite module categories can always be represented as the category of modules over some algebra object in the underlying tensor category \cite{etingof+et_al, douglas+schommer-pries+snyder}. 
We explicitly spell out its consequence for a finite bimodule category over the tensor category of comodules over a Hopf algebra.

In Section \ref{sec:twisted-center-via-monad} we express in Proposition \ref{prop:twisted-center-via-twisted-central-monad} the twisted center of a bimodule category over a general rigid linear tensor category $\A$ as the category of modules over an algebra in the enveloping tensor category $\A \boxtimes \A^\mopp$.
In Section \ref{sec:hopf-bimodules-for-twisted-center} this result, together with our explicit description of a bimodule category in terms of a bicomodule algebra in the Hopf algebra case at the end of Section \ref{sec:notation-and-background}, leads to Theorem \ref{thm:trace-of-bimodule-category-given-by-bicomodule-algebra}, the main result of this article:
the category-valued trace of such a bimodule category is equivalent to a category of (relative) Hopf bimodules.
The latter notion straightforwardly generalizes ordinary Hopf bimodules, which were studied in \cite{schauenburg,bespalov+drabant}.
Finally, at the end of Section \ref{sec:hopf-bimodules-for-twisted-center} we explain how the equivalence of Yetter-Drinfeld modules and (ordinary) Hopf bimodules \cite{schauenburg, bespalov+drabant} can be interpreted from the point of view of this article.

\subsection*{Acknowledgements}
The author is very grateful to Christoph Schweigert for supervising this project and for providing constant encouragement and advice.
The author would also like to thank Ehud Meir for helpful comments on a draft version of this paper.
The author is supported by the RTG 1670 ``Mathematics inspired by String theory and Quantum
Field Theory''.

\section{Notation, background and preliminaries} \label{sec:notation-and-background}

We start by declaring the finiteness conditions under which we work and by introducing the notation and conventions that we use.
We consider a fixed field $\kk$ which we assume to be algebraically closed.
A \emph{finite $\kk$-linear category} is a $\kk$-linear category that is equivalent to the $\kk$-linear abelian category $\lmod{A}$ of finite-dimensional left modules over a suitable finite-dimensional $\kk$-algebra $A$.
We will frequently use the fact that a $\kk$-linear functor between finite $\kk$-linear categories is right- or left-exact if and only if it has a right or left adjoint, respectively.
For a proof see e.g.\ \cite[Prop.~1.8]{douglas+schommer-pries+snyder}.
A \emph{finite $\kk$-linear tensor category} is a finite $\kk$-linear category with the structure of a tensor category (a.k.a.\ a monoidal category) such that the tensor product functor is $\kk$-bilinear and every object has left and right duals (i.e. it is rigid).
It is known that then the tensor product is exact in both arguments \cite{etingof+et_al}.
A right dual $x^\vee$ of an object $x$ in a tensor category $\A$ is an object $x^\vee$ equipped with an evaluation morphism $\ev_x : x^\vee \otimes x \to \II$ and a co-evaluation morphism $\coev_x : \II \to x \otimes x^\vee$ which satisfy appropriate zig-zag identities, where by $\II$ we denote the tensor unit object of $\A$.
This fixes our conventions for the left dual $\leftdual{x}$ as well.

A \emph{finite left module category over $\A$} is a finite $\kk$-linear category $\M$ together with a left action functor $\lact : \A \times \M \to \M$, which is $\kk$-bilinear and exact in the first argument (and hence in both due to rigidity of $\A$), and natural isomorphisms $((a \otimes b) \lact m \isoto a \lact (b \lact m))_{a,b \in \A, m \in \M}$ and $(\II \lact m \to m)_{m \in \M}$, called \emph{module constraints}, which satisfy coherence conditions analogous to the ones for the associativity and unit constraints of a tensor category.
A right module category is analogously defined.
A finite $\A_1$-$\A_2$-bimodule category $\M$ is a finite left $\A_1$-module category and a finite right $\A_2$-module category that is additionally equipped with a natural isomorphism $((a_1 \lact m) \ract a_2 \isoto a_1 \lact (m \ract a_2))_{a_1 \in \A_1, a_2 \in \A_2, m \in \M}$, called \emph{bimodule constraint}, satisfying additional coherence conditions.
An equivalent structure \cite{greenough}
is that of a finite left module category over the Deligne product $\A_1 \boxtimes \A_2^\mopp$, where $\A_2^\mopp$ has the opposite tensor product with respect to $\A_2$ and corresponding adapted associativity and unit constraints.
The existence and natural induced monoidal structure of the Deligne product $\A_1 \boxtimes \A_2^\mopp$ are guaranteed, respectively, by the finiteness of the categories and by the field $\kk$ being algebraically closed \cite{deligne}.
If such a structure has underlying action functor $\ogreaterthan : (\A_1 \boxtimes \A_2^\mopp) \times \M \lto \M$, then the corresponding left and right action functors are
\begin{equation} \label{eq:left-and-right-action-functors}
\begin{aligned}
\lact : \cat{A}_1 \times \cat{M} \lto \cat{M}, \quad (a_1,m) \lmapsto (a_1 \boxtimes \II_{\cat{A}_2}) \ogreaterthan m,	\\
\ract : \cat{M} \times \cat{A}_2 \lto \cat{M}, \quad (m,a_2) \lmapsto (\II_{\cat{A}_1} \boxtimes a_2) \ogreaterthan m.
\end{aligned}
\end{equation}
We remark that the bimodule constraints are obtained from the natural structure of the Deligne product functor $\boxtimes : \A_1 \times \A_2 \lto \A_1 \boxtimes \A_2$ as a tensor functor.

\subsection{The twisted center of a bimodule category}
We will arrive at our main result (Theorem \ref{thm:trace-of-bimodule-category-given-by-bicomodule-algebra}) by employing a construction of the category-valued trace of a finite bimodule category in terms of modules over a certain monad.
In fact, we show that it is canonically isomorphic to the twisted center, which gives \cite{fuchs+schaumann+schweigert} a realization of the category-valued trace.
Therefore we first recall:

\begin{definition} \label{def:twisted-center}
Let $\M$ be a finite $\A$-bimodule category for a finite $\kk$-linear tensor category $\A = (\A, \otimes, \II)$ with right duals.
Then the \emph{twisted center} (or, more precisely, the \emph{$?^{\vee \vee}$-twisted center}) $\twcenter(\cat{M})$ of $\cat{M}$ is the $\kk$-linear category whose objects are pairs $(m,\gamma_m)$ consisting of an object $m \in \cat{M}$ and a natural isomorphism
\[	\left( \gamma_m(a) : a \lact m \xlongrightarrow{\sim} m \ract a^{\vee \vee}	\right)_{a \in \cat{A}} \]
such that the following hexagonal diagram commutes for all $a, b \in \cat{A}$ and $m \in \cat{M}$:
\begin{center}
\begin{tikzpicture}
\matrix (m) [matrix of math nodes,row sep=1.5em,column sep=4em,minimum width=2em]{
		&	(a \otimes b) \lact m	&	m \ract (a \otimes b)^{\vee\vee}	&		\\
		a \lact (b \lact m)	&	&	&	(m \ract a^{\vee\vee}) \ract b^{\vee\vee}	\\
		&	a \lact (m \ract b^{\vee\vee}) &	(a \lact m) \ract b^{\vee\vee}	&	\\
};
\path[-stealth]
		(m-2-1)	edge node[above] {}	(m-1-2)
				edge node[below, xshift=-12pt] {$a \lact \gamma_m(b)$}			(m-3-2)
		(m-1-2)	edge node[below] {$\gamma_m(a \otimes b)$}			(m-1-3)
		(m-1-3)	edge node[above] {}	(m-2-4)
		(m-3-2)	edge node[right] {}	(m-3-3)
		(m-3-3)	edge node[below, xshift=12pt] {$\gamma_m(a) \ract b^{\vee\vee}$}	(m-2-4)
	;
\end{tikzpicture}
\end{center}
The morphisms $(m,\gamma_m) \to (m',\gamma_{m'})$ in $\twcenter(\cat{M})$ are those morphisms $m \to m'$ in $\cat{M}$ which are compatible with the natural isomorphisms $\gamma_m$ and $\gamma_{m'}$ in the obvious way.
The twisted center comes with an obvious exact $\kk$-linear forgetful functor $U : \twcenter(\M) \lto \M$.
\end{definition}
We consider the twisted variant of the Drinfeld center since it satisfies a universal property with respect to balanced functors from the bimodule category $\M$ into $\kk$-linear categories.
An \emph{$\A$-balanced functor} $F : \cat{M} \lto \cat{C}$ with \emph{balancing constraint $\beta^F$} is a $\kk$-linear functor $F : \cat{M} \lto \cat{C}$ together with a family of natural isomorphisms
\[ \left( \beta^F_{m,a} : F(m \ract a) \lisoto F(a \lact m) \right)_{m \in \cat{M}, a \in \cat{A}} \]
such that the following diagrams involving the coherence isomorphisms of the bimodule category commute for all objects $a, b \in \cat{A}$ and $m \in \cat{M}$:
\begin{center}
\begin{tikzpicture}
	\matrix (m) [matrix of math nodes,row sep=1.5em,column sep=4em,minimum width=2em]{
		F((m \ract a) \ract b))	&	F(m \ract (a \otimes b))		\\
		F(b \lact (m \ract a))	&	F((a \otimes b) \lact m)		\\
		F((b \lact m) \ract a)	&	F(a \lact (b \lact m))		\\
	};
	\path[-stealth]
		(m-1-1)	edge node[left] {$\beta^F_{m\ract a,b}$}	(m-2-1)
				edge node[above] {}			(m-1-2)
		(m-2-1)	edge node[right] {}			(m-3-1)
		(m-3-1)	edge node[above] {$\beta^F_{b\lact m, a}$}	(m-3-2)
		(m-1-2)	edge node[right] {$\beta^F_{m,a\otimes b}$}	(m-2-2)
		(m-2-2)	edge node[left] {}			(m-3-2)
	;
\end{tikzpicture}
\begin{tikzpicture}
	\matrix (m) [matrix of math nodes,row sep=4em,column sep=3em,minimum width=2em]{
		F(m \ract \II)	&	F(\II \lact m)		\\
		F(m)						&								\\
	};
	\path[-stealth]
		(m-1-1)	edge node[left] {}	(m-2-1)
				edge node[above] {$\beta^F_{m,\II}$}	(m-1-2)
		(m-1-2)	edge node[right] {}	(m-2-1)
	;
\end{tikzpicture}
\end{center}
The left-adjoint $U^\ell : \M \to \twcenter(\M)$ of the forgetful functor $U : \twcenter(\M) \to \M$, which exists by exactness of $U$, has a natural $\A$-balancing constraint $\beta$ determined by the following commuting diagram for all $a \in \A$, $m \in \M$ and $z \in \twcenter(\M)$.
\begin{equation} \label{eq:balancing-via-yoneda}
\begin{aligned}
\begin{tikzpicture}
	\matrix (m) [matrix of math nodes,row sep=1.5em,column sep=5em,minimum width=2em]{
		\Hom{\cat{M}}{m,\prescript{\vee}{}{a} \lact U(z)}	&	\Hom{\cat{M}}{m,U(z) \ract a^\vee}	\\
		\Hom{\cat{M}}{a \lact m,U(z)}	&	\Hom{\cat{M}}{m \ract a,U(z)}	\\
		\Hom{\twcenter(\cat{M})}{U^\ell(a \lact m), z}	&	\Hom{\twcenter(\cat{M})}{U^\ell(m \ract a),z}	\\
	};
	\path[-stealth]
		(m-1-1)	edge node[below] {$\gamma({\leftdual{a}})_*$} node[above] {$\sim$}	(m-1-2)
				edge node[left] {$\cong$}	(m-2-1)
		(m-2-1)	edge node[left] {$\cong$}	(m-3-1)
		(m-3-1)	edge node[below] {$\beta_{m,a}^*$} node[above] {$\sim$}	(m-3-2)
		(m-1-2)	edge node[left] {$\cong$}	(m-2-2)
		(m-2-2)	edge node[left] {$\cong$}	(m-3-2)
	;
\end{tikzpicture}
\end{aligned}
\end{equation}

This diagram also demonstrates the need for the double dual in the definition of the twisted center.
The universal property satisfied by this data $(\twcenter(\M),U^\ell,\beta)$ is that, for an arbitrary $\kk$-linear category $\C$, the $\kk$-linear functor
\begin{align*}
? \circ U^\ell : \funre{\kk}(\twcenter(\M), \C) &\lto \funbalre{\kk}(\M,\C), \\
F &\lmapsto F \circ U^\ell,
\end{align*}
from the category of right-exact $\kk$-linear functors from the twisted center to the category of right-exact $\A$-balanced $\kk$-linear functors from $\M$ is an adjoint equivalence.
This was shown in \cite{fuchs+schaumann+schweigert}, where this universal property is employed as the definition of the {category-valued trace}:
\begin{definition}[\cite{fuchs+schaumann+schweigert}]
A \emph{category-valued trace} $(\cattr(\M),B_{\M},\Psi_{\M}, \varphi_{\M}, \kappa_{\M})$ of an $\A$-bimodule category $\M$ for a finite $\kk$-linear tensor category $\A$ is an abelian $\kk$-linear category $\cattr(\M)$ together with an $\A$-balanced functor $B_{\M} : \M \lto \cattr(\M)$ such that for every abelian $\kk$-linear category $\C$ the functor
\begin{align*}
\Phi_{\M;\C} :  \funre{\kk}(\cattr(\M), \C) &\lto \funbalre{\kk}(\M,\C), \\
F &\lmapsto F \circ B_{\M},
\end{align*}
is an equivalence of $\kk$-linear categories, and together with the following structure:
for any $\kk$-linear category $\C$ a specified quasi-inverse
\( \Psi_{\M;\C} \colon \funbalre{\kk}(\M,\C) \to \funre{\kk}(\cattr(\M), \C) \)
and an adjoint equivalence between $\Phi_{\M;\C}$ and $\Psi_{\M;\C}$ formed by $\varphi_{\M;\C} : \id{} \to \Phi_{\M;\C} \Psi_{\M;\C}$ and $\kappa_{\M;\C} : \Psi_{\M;\C} \Phi_{\M;\C} \to \id{}$.
(These choices always exist.)
\end{definition}
In the case of the $\A$-bimodule category $\M \boxtimes \N$ given by the Deligne product of a right $\A$-module category $\M$ and a left $\A$-module category $\N$, the universal property for the category-valued trace is precisely the one for the relative Deligne product $\M \boxtimes_{\A} \N$ \cite{douglas+schommer-pries+snyder}.

\subsection{Module categories in terms of algebras} \label{subsec:module-categories-via-algebras}
Any algebra $A$ in a monoidal category $\A$ gives rise to a left module category $\rmodin{\A}{A}$ over $\A$, where $\rmodin{\A}{A}$ is the $\kk$-linear category of right $A$-modules in $\A$ \cite{etingof+et_al}.
For a proof of the following converse to this statement see for example \cite[Theorem~2.24, pf.~of~Lemma~2.25]{douglas+schommer-pries+snyder}.

\begin{theorem} \label{thm:module-category-via-algebra}
Let $\A$ be a rigid finite $\kk$-linear tensor category and let $\cat{M}$ be a finite left module category over $\A$.
Then there exists an algebra $A$ in $\A$ together with an equivalence of $\A$-module categories $\cat{M} \simeq \rmodin{\cat{A}}{A}$.
\end{theorem}
\begin{remark} \label{rem:thm-module-category-via-algebra}
The algebra $A \in \A$ can be obtained as $A = \intHom{m,m} \in \A$ and the equivalence as $\intHom{m,?}$ for any projective generator $m \in \M$.
Here $\intHomwoarg: \M^\opp \times \M \to \A$ is the internal hom functor of the module category $\M$, which is determined by $\intHom{m,?}$ being the right adjoint to $?  \lact m $ for every $m \in \M$.
It exists for finite categories because the action functor is exact and therefore has, in particular, a right adjoint.
\end{remark}

Let $H_1$ and $H_2$ be finite-dimensional Hopf algebras over $\kk$ and denote by $\lcomod{H}$ for any Hopf algebra $H$ the $\kk$-linear ten\-sor category of $H$-comodules.
Now we apply Theorem \ref{thm:module-category-via-algebra} to $(\lcomod{H_1})$-$(\lcomod{H_2})$-bimodule categories.
An $(\lcomod{H_1})$-$(\lcomod{H_2})$-bimodule category is equivalently a left module category over \[ \lcomod{H_1} \boxtimes (\lcomod{H_2})^\mopp \simeq \lcomod{H_1 \otimes H_2^\opp} \cong \bicomod{H_1}{H_2^{\opcop}} \]
where the latter denotes the $\kk$-linear tensor category of finite-dimensional $H_1$-$H_2^\opcop$-bi\-co\-mod\-ules with tensor product over $\kk$.
$H_2^\opcop$ is the Hopf algebra obtained from $H_2$ by flipping both multiplication and co-multiplication.
(Actually we could simplify our description by using that $H_2$ is isomorphic to $H_2^\opcop$ as a Hopf algebra via the antipode, which is invertible for finite-dimensional Hopf algebras.
However, this is not canonical since any other odd power of the antipode also provides such an isomorphism.)

According to Theorem \ref{thm:module-category-via-algebra} an arbitrary finite left $(\bicomod{H_1}{H_2^{\opcop}})$-module category arises from an algebra in the tensor category $\bicomod{H_1}{H_2^{\opcop}}$, which is also called an \emph{$H_1$-$H_2^\opcop$-bicomodule algebra}.
Hence, let $B_{12} \in \bicomod{H_1}{H_2^{\opcop}}$ be an $H_1$-$H_2^\opcop$-bicomodule algebra and let us consider the $\kk$-linear category $\genHopf{H_1}{H_2^\opcop}{}{B_{12}}$, which is a left $(\bicomod{H_1}{H_2^{\opcop}})$-module category in a canonical way.
Considering it as an $(\lcomod{H_1})$-$(\lcomod{H_2})$-bimodule category we write down explicitly the underlying left and right action functors, as in \eqref{eq:left-and-right-action-functors}:
\begin{itemize}
\item The left $(\lcomod{H_1})$-action functor is
\begin{equation*}
\begin{aligned}
\lact : \genHopf{H_1}{}{}{} \times \genHopf{H_1}{H_2^\opcop}{}{B_{12}} &\lto \genHopf{H_1}{H_2^\opcop}{}{B_{12}},	\\
(X_1,M) &\lmapsto X_1 \lact M := (X_1 \boxtimes \kk) \ogreaterthan M,
\end{aligned}
\end{equation*}
where $(X_1 \boxtimes \kk) \ogreaterthan M \cong X_1 \otimes M$ as a vector space, with bicomodule structure
\begin{equation*}
\begin{aligned}
\delta_{X_1 \lact M} : X_1 \otimes M &\lto H_1 \otimes X_1 \otimes M \otimes H_2^\opcop, \\
x \otimes m &\lmapsto x_{(-1)} m_{(-1)} \otimes x_{(0)} \otimes m_{(0)} \otimes m_{(1)},
\end{aligned}
\end{equation*}
using Sweedler notation $X_1 \ni x \mapsto x_{(-1)} \otimes x_{(0)} \in H_1 \otimes X_1$ for a left $H_1$-comodule and $M \ni m \mapsto m_{(-1)} \otimes m_{(0)} \otimes m_{(1)} \in H_1 \otimes M \otimes H_2^\opcop$ for an $H_1$-$H_2^\opcop$-bicomodule.
\item The right $(\lcomod{H_2})$-action functor is
\begin{equation*}
\begin{aligned}
\ract : \genHopf{H_1}{H_2^\opcop}{}{B_{12}} \times \genHopf{H_2}{}{}{} &\lto \genHopf{H_1}{H_2^\opcop}{}{B_{12}},	\\     
(M,X_2) &\lmapsto M \ract X_2 := (\kk \boxtimes X_2) \ogreaterthan M,
\end{aligned}
\end{equation*}
where $(\kk \boxtimes X_2) \ogreaterthan M \cong M \otimes X_2$ as a vector space, with bicomodule structure
\begin{equation*}
\begin{aligned}
\delta_{M \ract X_2} : M \otimes X_2 &\lto H_1 \otimes M \otimes X_2 \otimes H_2^\opcop, \\
m \otimes x &\lmapsto m_{(-1)} \otimes m_{(0)} \otimes x_{(0)} \otimes m_{(1)} x_{(-1)},
\end{aligned}
\end{equation*}
where, when we write $m_{(1)} x_{(-1)}$, we are using the multiplication in $H_2$ and not in $H_2^\opcop$.
\end{itemize}

\section{The twisted center as the category of modules over a monad} \label{sec:twisted-center-via-monad}

In this section we prove a result (Theorem \ref{thm:category-valued-trace-as-monad-modules}) which holds for bimodule categories over a general $\kk$-linear tensor category $\A$.
In the subsequent section we will then specialize to the case of comodules over a Hopf algebra, $\A = \lcomod{H}$.

In \cite{day+street} it is shown that the Drinfeld center of a tensor category $\A$ is equivalent to the category of modules over the monad $x \mapsto \int^{a \in \cat{A}} a \otimes x \otimes \leftdual{a}$ on $\A$, called the central monad.
For the definition of a monad see \cite{mac_lane}.
A \emph{module} over a monad $T: \C \to \C$ (also called a $T$-algebra) is an object $M \in \C$ together with a morphism $\rho : T(M) \to M$ satisfying associativity and unitality laws with respect to the monad $T$.
We denote by $\modmon{\C}{T}$ the category of such modules.
Here we straightforwardly generalize the result of \cite{day+street} to arbitrary bimodule categories $\M$ over $\A$ and, since we are interested in the twisted center, we instead use the twisted central monad on $\M$ given by $m \mapsto \int^{a \in \cat{A}} a \lact m \ract a^\vee$, involving the right dual instead of the left dual.
This does not require any essential changes to the proof of the original result.

We define the twisted central monad on $\M$ by defining an algebra $\tild{A}$ in the enveloping tensor category $\A \boxtimes \A^\mopp$.
As an object define $\tild{A} := \int^{a \in \A} a \boxtimes a^\vee \in \A \boxtimes \A^\mopp$.
If $\A$ is finite then this coend exists (e.g. \cite[Theorem~3.6]{shimizu}) because its underlying functor is $\kk$-linear and exact in both arguments.
To define a multiplication morphism $ \mu_{\tild{A}} : \tild{A} \otimes \tild{A} \lto \tild{A} $ denote the universal dinatural transformation for $\tild{A}$ by $j$ and note that the dinatural transformation given for $a,b \in \A$ by
\begin{equation*} \label{eq:fubini-coend}
	j^{(2)}(a,b) : (a \otimes b) \boxtimes (a \otimes b)^\vee \cong (a \boxtimes a^\vee) \otimes (b \boxtimes b^\vee) \xlongrightarrow{j(a) \otimes j(b)} \tild{A} \otimes \tild{A}
\end{equation*}
makes $\tild{A} \otimes \tild{A}$ into a coend $\int^{(a,b) \in \A \times \A} (a \otimes b) \boxtimes (a \otimes b)^\vee$ by the Fubini theorem for coends and by a compatibility with tensor products.
Then the multiplication $\mu_{\tild{A}}$ is defined as the unique morphism such that for all $a,b \in \cat{A}$
\begin{equation} \label{eq:algebra-structure-on-coend}
\mu_{\tild{A}} \circ j^{(2)}(a,b) = j(a \otimes b)
\end{equation}
The unit for this multiplication is given by $j(\II) : \II \boxtimes \II \to \tild{A}$.
For a proof of associativity of the algebra we just defined compare with a similar algebra defined in \cite[Section~4.3]{shimizu}.

\begin{definition}
Let $\M$ be a finite bimodule category over a rigid finite $\kk$-linear tensor category $\A$.
Then we call the monad induced by the above algebra $(\tilde{A}, \mu_{\tild{A}})$ in $\A \boxtimes \A^\mopp$, with underlying endofunctor
\begin{align*}
	\widetilde{Z} \colon \cat{M}  &\lto \cat{M},	\\
	m &\lmapsto \tilde{A} \ogreaterthan m \cong \int^{a \in \cat{A}} (a \lact m) \ract a^\vee,
\end{align*}
and multiplication $\tilde{\mu} := \mu_{\tild{A}} \ogreaterthan \id{} : (\tild{A}\otimes \tild{A}) \ogreaterthan ? \Longrightarrow \tild{A} \ogreaterthan ?$,
the \emph{twisted central monad} on $\M$.
\end{definition}
\begin{remark}
Here we have used the natural isomorphism $\tild{A} \ogreaterthan m \cong \int^{a \in \cat{A}} (a \lact m) \ract a^\vee$, which follows from the fact that the action functor $? \ogreaterthan m : \A \to \M$ is cocontinuous because it is by assumption right-exact.
Let us denote by $i_m := j \ogreaterthan \id{m}$ the universal dinatural transformation for $\tild{A} \ogreaterthan m \cong \int^{a \in \cat{A}} (a \lact m) \ract a^\vee$.
\end{remark}

Then we have:

\begin{proposition} \label{prop:twisted-center-via-twisted-central-monad}
Let $\cat{A}$ be a rigid finite $\kk$-linear tensor category and let $\cat{M}$ be a finite $\cat{A}$-bimodule category.
Then there is a canonical equivalence of $\kk$-linear categories between the category $\modmon{\cat{M}}{\tild{Z}}$ of modules over the twisted central monad $\tild{Z}$ on $\cat{M}$ and the twisted center $\twcenter(\cat{M})$ of $\cat{M}$. This equivalence fits into a strictly commutative triangle with the forgetful functors of $\modmon{\cat{M}}{\tild{Z}}$ and $\twcenter(\cat{M})$ to $\cat{M}$:
\begin{center}
\begin{tikzpicture}
\matrix (m) [matrix of math nodes,row sep=2em,column sep=2em,minimum width=2em]{
	\modmon{\cat{M}}{\tild{Z}}	&		&	\twcenter(\cat{M})	\\
								&	\M	&						\\
};
\path[-stealth]
	(m-1-1)	edge node[above] {$\cong$}	(m-1-3)
			edge node[below left] {}	(m-2-2)
	(m-1-3)	edge node[below right] {}	(m-2-2)
;
\end{tikzpicture}
\end{center}
\end{proposition}
\begin{proof}
Since the proof is analogous to the one for the center of a tensor category \cite{day+street} we will restrict ourselves to providing the mutually inverse category equivalences, which we will need to refer to in the proof of Theorem \ref{thm:category-valued-trace-as-monad-modules}.

In fact, for any object $m \in \M$ there is a bijection between $\tild{Z}$-module structures $\rho : \tild{Z}(m)=\int^{a \in \cat{A}} (a \lact m) \ract a^\vee \to m$ and natural isomorphisms $\left(	a \lact m \lto m \ract a^{\vee\vee} \right)_{a \in \cat{A}}$ satisfying the hexagon axiom.
We specify this bijection:

Given a morphism $\rho : \tild{Z}(m) \lto m$ we define a natural transformation $\left(	a \lact m \lto m \ract a^{\vee\vee} \right)_{a \in \cat{A}}$ by the composition:
\begin{center}
\begin{tikzpicture}
	\matrix (m) [matrix of math nodes,row sep=2.5em,column sep=3.9em,minimum width=2em]{
		a \lact m	&	&	m \ract a^{\vee\vee}									\\
		(a \lact m) \ract \II	&	(a \lact m) \ract (a^\vee \otimes a^{\vee\vee}) \cong ((a \lact m) \ract a^\vee) \ract a^{\vee\vee}
								&	\tild{Z}(m) \ract a^{\vee\vee}						\\
	};
	\path[-stealth]
		(m-1-1)	edge node[above] {}											(m-1-3)
				edge node[left] {$\cong$}									(m-2-1)
		(m-2-1)	edge node[below] {\scalebox{0.8}{$\id{a \lact m} \ract \coev_{a^\vee}$}}		(m-2-2)
		(m-2-2)	edge node[below] {\scalebox{0.8}{$i_m(a) \ract \id{a^{\vee\vee}}$}}				(m-2-3)
		(m-2-3)	edge node[left] {$\rho \ract \id{a^{\vee\vee}}$}				(m-1-3)
	;
\end{tikzpicture}
\end{center}
The unnamed isomorphisms in this diagram are the unique isomorphisms built from the unit, associativity and (bi-)module constraints of the tensor category $\cat{A}$ and the bimodule category
$\cat{M}$.
From now on we will assume for simplicity that $\cat{A}$ and $\cat{M}$ are strict as tensor and bimodule category respectively, that is we will not explicitly include these isomorphisms anymore.

Conversely, let $\left( \gamma_m(a) : a \lact m \lto m \ract a^{\vee\vee} \right)_{a \in \cat{A}}$
be a natural transformation.
Then, using the universal property of the coend $(\tild{Z}(m), i_m)$, we define a morphism $\tild{Z}(m) \lto m$ by demanding it to uniquely make the following diagram commute for every $a \in \cat{A}$:
\begin{center}
\begin{tikzpicture}
	\matrix (m) [matrix of math nodes,row sep=3em,column sep=10em,minimum width=2em]{
		\tild{Z}(m)					&	m									\\
		(a \lact m) \ract a^\vee	&	(m \ract a^{\vee\vee}) \ract a^\vee	\\
	};
	\path[-stealth]
		(m-1-1)	edge node[above] {}				(m-1-2)
		(m-2-1)	edge node[right] {$i_m(a)$}		(m-1-1)
		(m-2-1)	edge node[below] {$\gamma_m(a) \ract \id{a^\vee}$}		(m-2-2)
		(m-2-2)	edge node[right] {$\id{m} \ract \ev_{a^\vee}$}			(m-1-2)
	;
\end{tikzpicture}
\end{center}

This way one obtains a bijection between $\text{Hom}_\M (\tild{Z}(m),m)$ and $\Nat{? \lact m, m\ \ract\ ?^{\vee\vee}}$.
Finally, it is straightforward to show that a morphism in $\text{Hom}_\M (\tild{Z}(m),m)$ is a $\tild{Z}$-module structure if and only if the corresponding natural transformation in $\Nat{? \lact m, m\ \ract\ ?^{\vee\vee}}$ under this bijection is invertible and satisfies the hexagon axiom of Definition \ref{def:twisted-center}.
\end{proof}

We close this section by summarising the above result and \cite[Theorem~3.3(i)]{fuchs+schaumann+schweigert} in the following theorem.

\begin{theorem} \label{thm:category-valued-trace-as-monad-modules}
Let $\cat{A}$ be a rigid, finite $\kk$-linear tensor category and let $\cat{M}$ be a finite $\cat{A}$-bimodule category.
Then the category $\lmodin{\cat{M}}{\tild{A}}$ of modules over the monad on $\cat{M}$ given by the algebra $\tild{A} = \int^{a \in \A} a \boxtimes a^\vee \in \cat{A} \boxtimes \A^\mopp$ is a category-valued trace of $\cat{M}$ together with the balanced functor
\begin{align*}
	(\tild{A} \ogreaterthan ?) : \cat{M} &\lto \lmodin{\cat{M}}{\tild{A}},	\\
	m &\longmapsto (\tild{A} \ogreaterthan m, \mu_{\tild{A}} \ogreaterthan \id{m}),
\end{align*}
whose balancing $\beta$ is determined by the commutative diagram
\begin{equation} \label{eq:balancing-via-yoneda-module-case}
\begin{tikzpicture}
	\matrix (m) [matrix of math nodes,row sep=3em,column sep=10em,minimum width=2em]{
		\Hom{\cat{M}}{m,( ({}^\vee a \lact U(M))\ract a ) \ract a^\vee}
			&	\Hom{\cat{M}}{m,(\tild{A} \ogreaterthan U(M)) \ract a^\vee}	\\
		\Hom{\cat{M}}{m,{}^\vee a \lact U(M)}	&	\Hom{\cat{M}}{m,U(M) \ract a^\vee}	\\
		\Hom{\cat{M}}{a \lact m,U(M)}	&	\Hom{\cat{M}}{m \ract a,U(M)}	\\
		\Hom{\lmodin{\cat{M}}{\tild{A}}}{\tild{A} \ogreaterthan (a \lact m), M}
			&	\Hom{\lmodin{\cat{M}}{\tild{A}}}{\tild{A} \ogreaterthan (m \ract a), M}	\\
	};
	\path[-stealth]
		(m-1-1)	edge node[below] {$((j({}^\vee a)\ogreaterthan \id{U(M)}) \ract \id{a^\vee})_*$}	(m-1-2)
		(m-2-1)	edge node[right] {$((\id{{}^\vee\! a \lact U(M)}) \ract \coev_{a})_*$}	(m-1-1)
		(m-1-2)	edge node[left] {$(\rho_M \ract \id{a^\vee})_*$}			(m-2-2)
		(m-3-1)	edge node[left] {$\cong$}			(m-2-1)
		(m-4-1)	edge node[left] {$\cong$}			(m-3-1)
		(m-4-1)	edge node[above] {$\beta_{m,a}^*$}	(m-4-2)
		(m-2-2)	edge node[left] {$\cong$}			(m-3-2)
		(m-3-2)	edge node[left] {$\cong$}			(m-4-2)
	;
\end{tikzpicture}
\end{equation}
for $a\!\in\!\cat{A}, m\!\in\!\cat{M}$ and $M\!\in\!\lmodin{\cat{M}}{\tild{A}}$.
\end{theorem}

\begin{proof}
From \cite[Theorem~3.3(i)]{fuchs+schaumann+schweigert} it follows that the universal balanced functor for the twisted center $\twcenter(\cat{M})$ is the left adjoint of the forgetful functor $\twcenter(\cat{M}) \lto \cat{M}$.
Its balancing is determined by the commutative diagram (\ref{eq:balancing-via-yoneda}) via the Yoneda lemma.
In view of Proposition \ref{prop:twisted-center-via-twisted-central-monad} we can take the left adjoint of the forgetful functor $U : \lmodin{\cat{M}}{\tild{A}} \lto \cat{M}$ as the universal balanced functor into $\lmodin{\cat{M}}{\tild{A}} \stackrel{\text{def}}{=} \modmon{\M}{(\tild{A} \ogreaterthan ?)}$ in order to furnish the latter with the structure of category-valued trace of $\cat{M}$.
It is easy to see that the \emph{induction} functor (or \emph{free} functor) given in the statement of the theorem is the said left adjoint.
The morphisms in diagram (\ref{eq:balancing-via-yoneda-module-case}) are obtained from diagram (\ref{eq:balancing-via-yoneda}) by using the canonical category isomorphism given in the proof of Proposition \ref{prop:twisted-center-via-twisted-central-monad}.
\end{proof}

\section{Main result: The twisted center in the Hopf algebra case} \label{sec:hopf-bimodules-for-twisted-center}

In the next step we spell out Theorem \ref{thm:category-valued-trace-as-monad-modules} explicitly in the special case of the situation described in Subsection \ref{subsec:module-categories-via-algebras}.
We thus consider the tensor category $\cat{A} = \lcomod{H}$ of finite-dimensional comodules over a finite-dimensional Hopf algebra $H$ and the $(\lcomod{H})$-bimodule category $\genHopf{H}{H^\opcop}{}{B}$ induced by an $H$-$H^\opcop$-bicomodule algebra $B$, as explained in Subsection \ref{subsec:module-categories-via-algebras}.

In this concrete case it is possible to obtain an explicit description of the algebra \( \tild{A} = \int^{X \in \lcomod{H}} X \boxtimes X^\vee . \)
Let $\tild{H}$ be the algebra underlying the Hopf algebra $H$ together with the following $H$-$H^\opcop$-bi\-co\-mo\-dule structure:
\begin{align*}
	\delta^\ell_{\tild{H}} : \tild{H} & \lto H \otimes \tild{H}, & \delta^r_{\tild{H}} : \tild{H} &\lto \tild{H} \otimes H^\opcop, \\
	x &\longmapsto x_{(1)} \otimes x_{(2)}, &
	x &\longmapsto x_{(1)} \otimes S^{-1}(x_{(2)}) .
\end{align*}
This way $\tild{H}$ is an algebra in the tensor category $\genHopf{H}{H^\opcop}{}{}$. We have:

\begin{lemma} \label{lem:Htilde-as-canonical-algebra}
The $H$-$H^\opcop$-bi\-co\-mo\-dule $\tild{H}$ together with the dinatural family 
\[ \left( j_X := (\id{H} \otimes (\ev_{X} \circ \tau_{X,X^\vee}) ) \circ (\delta_X \otimes \id{X^\vee}) : X \boxtimes X^\vee \lto \tild{H} \right)_{X \in \lcomod{H}} \]
is a coend for the functor $\lcomod{H} \times (\lcomod{H})^\opp \ni (X,Y) \longmapsto X \boxtimes Y^\vee \in \genHopf{H}{H^\opcop}{}{}$, where by $\tau$ we denote the standard symmetric braiding of the underlying vector spaces and by $\boxtimes : \lcomod{H} \times (\lcomod{H})^\mopp \lto \genHopf{H}{H^\opcop}{}{}$ we denote the universal right-exact bilinear functor for the Deligne product $\genHopf{H}{H^\opcop}{}{} \simeq  \lcomod{H} \boxtimes (\lcomod{H})^\mopp$.
Furthermore, the algebra structure on $\tild{H}$ inherited from $H$ coincides with the algebra structure on the coend determined by (\ref{eq:algebra-structure-on-coend}).
\end{lemma}

\begin{proof}
The dinaturality property of the proposed coend follows from the fact that $H$-co\-module morphisms commute with the $H$-co-action.
It remains to show the universality of this co-wedge.
For this let $(T,\alpha)$ be an arbitrary co-wedge, \[ \alpha := (\alpha_X : X \boxtimes X^\vee \lto T)_{X \in \lcomod{H}} . \]
We have to find a unique morphism of co-wedges $\phi : \tild{H} \lto T$.
Consider the regular $H$-comodule $H$ and note that for any $h \in H$ we have
\[ j_H(h \otimes \eps) = h_{(1)} \eps(h_{2}) = h,  \]
where $\eps : H \lto \kk$ denotes the co-unit of $H$.
Thus the required morphism $\phi : \tild{H} \lto T$ of co-wedges is completely determined by its property that $\phi \circ j_X = \alpha_X$ for all $X \in \lcomod{H}$:
\[ \phi(h) = \phi(j_H(h \otimes \eps)) \stackrel{!}{=} \alpha_H(h\otimes\eps) . \]
It remains to show that this indeed gives a well-defined morphism of co-wedges:
The fact that it is a morphism of $H$-$H^\opcop$-co\-modules uses the dinaturality of $\alpha$.
Further we have to show that $\phi \circ j_X = \alpha_X$ for all $X \in \lcomod{H}$.
We have $\phi(j_X(x\otimes\varphi)) = \alpha_H(j_X(x\otimes\varphi)\otimes \eps)$ for all $x\otimes\varphi \in X\boxtimes X^\vee$.
Hence, what we have to show is that $\alpha_H(j_X(x\otimes\varphi)\otimes \eps) = \alpha_X(x \otimes \varphi)$.
This follows from the dinaturality of $\alpha$ applied to the $H$-comodule morphism $j_X(-\otimes \varphi) : X \lto H$.
Indeed, for this it suffices to show that $(j_X(-\otimes \varphi))^\vee(\eps) = \varphi$, where $(j_X(-\otimes \varphi))^\vee : H^\vee \lto X^\vee$ is the dual morphism, which follows easily from the definition of $j$.

To prove the second claim we have to show that the following diagram commutes for all $X,Y \in \lcomod{H}$.
\begin{center}
	\begin{tikzpicture}
	\matrix (m) [matrix of math nodes,row sep=3em,column sep=4em,minimum width=2em]{
		\tild{H} \otimes \tild{H}	&	\tild{H}		\\
		(X \boxtimes X^\vee) \otimes (Y \boxtimes Y^\vee)	&	(X \otimes Y) \boxtimes (Y^\vee \otimes X^\vee)	\\
	};
	\path[-stealth]
		(m-1-1)	edge node[above] {$\mu$}				(m-1-2)
		(m-2-1)	edge node[left] {$j(X) \otimes j(Y)$}	(m-1-1)
		(m-2-1)	edge node[above] {$\cong$}				(m-2-2)
		(m-2-2)	edge node[right] {$j(X \otimes Y)$}		(m-1-2)
	;
	\end{tikzpicture}
\end{center}
It is straightforward to verify that this is the case by inserting the definition of $j$.
Indeed, $j(X \otimes Y)$ involves the co-action of the tensor product $X \otimes Y$, which is given by individual co-action of $X$ and $Y$ and then multiplying in $H$.
\end{proof}

Together with Lemma \ref{lem:Htilde-as-canonical-algebra}, Theorem \ref{thm:category-valued-trace-as-monad-modules} now specialises to the following main result of this article.
For $H$ and $H'$ bialgebras and $A$ and $B$ $H$-$H'$-bicomodule algebras we denote by $\genHopf{H}{H'}{A}{B}$ the $\kk$-linear category of $A$-$B$-bimodule objects in the tensor category of $H$-$H'$-bicomodules.
We call such bimodules \emph{relative Hopf bimodules} (cf. \cite{schauenburg,bespalov+drabant} for the non-relative case $H=H'=A=B$, and cf. \cite{montgomery} for the one-sided case $H=A=\kk$).
Now we can state:

\begin{theorem} \label{thm:trace-of-bimodule-category-given-by-bicomodule-algebra}
Let $H$ be a finite-dimensional Hopf algebra over $\kk$ and let $B$ be a finite-dimensional $H$-$H^\opcop$-bico\-module algebra.
Then the category-valued trace of the bimodule category $\genHopf{H}{H^\opcop}{}{B}$ is realised by the $\kk$-linear category $\genHopf{H}{H^\opcop}{\tild{H}}{B}$ of relative $\tild{H}$-$B$-Hopf-bimodules together with the universal balanced functor
\[	\tild{H} \ogreaterthan ? = \tild{H} \otimes ? : \genHopf{H}{H^\opcop}{}{B} \lto \genHopf{H}{H^\opcop}{\tild{H}}{B},	\]
whose balancing is
\begin{align*}
	\beta_{M,X} : \tild{H} \otimes (M \ract X) &\xlongrightarrow{\sim} \tild{H} \otimes (X \lact M),	\\
	h \otimes m \otimes x &\longmapsto (h S(x_{(-1)})) \otimes x_{(0)} \otimes m,
\end{align*}
for $X \in \lcomod{H}$ and $M \in \genHopf{H}{H^\opcop}{}{B}$.
\end{theorem}
\begin{proof}
The proof is done by applying Theorem \ref{thm:category-valued-trace-as-monad-modules}:
According to Theorem \ref{thm:category-valued-trace-as-monad-modules} the category-valued trace of the given bimodule category is \[ \lmodin{\left(\genHopf{H}{H^\opcop}{}{B}\right)}{\tild{H}}, \] the category of modules over the monad on $\genHopf{H}{H^\opcop}{}{B}$ given by the algebra $\tild{H} \in \genHopf{H}{H^\opcop}{}{}$.
This category is straightforwardly isomorphic to $\genHopf{H}{H^\opcop}{\tild{H}}{B}$.

To prove the claim about the balancing we let $f \in \Hom{\genHopf{H}{H^\opcop}{\tild{H}}{B}}{\tild{H} \otimes (X \lact M),N}$ for $X \in \lcomod{H}, M \in \genHopf{H}{H^\opcop}{}{B}$ and $N \in \genHopf{H}{H^\opcop}{\tild{H}}{B}$ and chase it through the diagram \eqref{eq:balancing-via-yoneda-module-case}, which, for convenience, we write out again for the present special situation:
\begin{center}
\begin{tikzpicture}
\matrix (m) [matrix of math nodes,row sep=3em,column sep=8em,minimum width=2em]{
	\Hom{\genHopf{H}{H^\opcop}{}{B}}{M,( ({}^\vee X \lact N)\ract X ) \ract X^\vee}
		&	\Hom{\genHopf{H}{H^\opcop}{}{B}}{M,(\tild{H} \otimes N) \ract X^\vee}	\\
	\Hom{\genHopf{H}{H^\opcop}{}{B}}{M,{}^\vee X \lact N}	&	\Hom{\genHopf{H}{H^\opcop}{}{B}}{M, N \ract X^\vee}	\\
	\Hom{\genHopf{H}{H^\opcop}{}{B}}{X \lact M, N}	&	\Hom{\genHopf{H}{H^\opcop}{}{B}}{M \ract X, N}	\\
	\Hom{\genHopf{H}{H^\opcop}{\tild{H}}{B}}{\tild{H} \otimes (X \lact M), N}
		&	\Hom{\genHopf{H}{H^\opcop}{\tild{H}}{B}}{\tild{H} \otimes (M \ract X), N}	\\
};
\path[-stealth]
	(m-1-1)	edge node[below] {\scalebox{0.7}{$((j({}^\vee X)\ogreaterthan \id{N}) \ract \id{X^\vee})_*$}}	(m-1-2)
	(m-2-1)	edge node[right] {$((\id{{}^\vee\! X \lact N}) \ract \coev_{X})_*$}	(m-1-1)
	(m-1-2)	edge node[right] {$(\rho_N \ract \id{X^\vee})_*$}			(m-2-2)
	(m-3-1)	edge node[right] {$\cong$}			(m-2-1)
	(m-4-1)	edge node[right] {$\cong$}			(m-3-1)
	(m-4-1)	edge node[above] {$\beta_{M,X}^*$}	(m-4-2)
	(m-2-2)	edge node[right] {$\cong$}			(m-3-2)
	(m-3-2)	edge node[right] {$\cong$}			(m-4-2)
;
\end{tikzpicture}
\end{center}

In the following computation each mapping arrow corresponds to exactly one morphism in the above diagram.
We start in the bottom left corner and go via the top to arrive in the bottom right corner.
The composition is equal to $\beta_{M,X}^* = ? \circ \beta_{M,X}$ by the commutativity of the diagram.

\begin{align*}
f &\lmapsto f \circ (1_H \otimes \id{X \otimes M}) \\
&\lmapsto (\id{\leftdual{X}}\otimes(f \circ (1_H \otimes \id{X \otimes M})))\circ(\lcoev_{X}\otimes \id{M}) \\
&\lmapsto (\id{\leftdual{X}}\otimes(f \circ (1_H \otimes \id{X \otimes M}))\otimes \coev_X)\circ(\lcoev_{X}\otimes \id{M}) \\
&\lmapsto ((j(\leftdual)\ogreaterthan \id{N})\otimes \id{X^\vee}) \circ (\id{\leftdual{X}}\otimes(f \circ (1_H \otimes \id{X \otimes M}))\otimes \coev_X)\circ(\lcoev_{X}\otimes \id{M}) \\
& \ \ \ \ \ \ \stackrel{(\ast)}{=} (( (\id{H}\otimes (\ev_X \circ \tau_{\leftdual{X},X}))\circ (\delta_{\leftdual{X}}\otimes \tau_{N,X}) )\otimes \id{X^\vee}) \\
&  \ \ \ \ \ \ \ \ \ \ \circ (\id{\leftdual{X}}\otimes(f \circ (1_H \otimes \id{X \otimes M}))\otimes \coev_X)\circ(\lcoev_{X}\otimes \id{M}) \\
& \ \ \ \ \ \ \stackrel{(\dagger)}{=} (S \otimes (f\circ (1_H \otimes \id{X \otimes M})\otimes \id{X^\vee} )) \circ (\delta_X \otimes \tau_{X^\vee,M})\circ (\coev_X \otimes \id{M}) \\
&\lmapsto (\rho_N \circ (S \otimes (f\circ (1_H \otimes \id{X \otimes M}))\otimes \id{X^\vee} )) \circ (\delta_X \otimes \tau_{X^\vee,M})\circ (\coev_X \otimes \id{M}) \\
&\lmapsto (\id{N}\otimes \ev_X) \\
&  \ \ \ \ \ \ \ \ \ \ \circ (((\rho_N \circ (S \otimes (f\circ (1_H \otimes \id{X \otimes M}))\otimes \id{X^\vee} )) \circ (\delta_X \otimes \tau_{X^\vee,M})\circ (\coev_X \otimes \id{M}))\otimes \id{X}) \\
& \ \ \ \ \ \ \stackrel{(\diamond)}{=} \rho_N \circ (S \otimes (f\circ (1_H \otimes \id{X \otimes M}))) \circ (\delta_X \otimes \id{M}) \circ \tau_{M,X} \\
&\lmapsto \rho_N \circ (\id{H} \otimes (\rho_N \circ (S \otimes (f\circ (1_H \otimes \id{X \otimes M}))) \circ (\delta_X \otimes \id{M}) \circ \tau_{M,X})) \\
& \ \ \ \ \ \ \stackrel{(\#)}{=} f \circ ((\mu_H \circ (\id{H}\otimes S))\otimes \id{X \otimes M})\circ (\id{H} \otimes \delta_X \otimes \id{M}) \circ (\id{H} \otimes \tau_{M,X}) \\
& \ \ \ \ \ \ \ = f \circ (\mu_H \otimes \id{X\otimes M}) \circ (\id{H}\otimes S\otimes \id{X \otimes M})\circ (\id{H} \otimes \delta_X \otimes \id{M}) \circ (\id{H} \otimes \tau_{M,X})
\end{align*}

In the step $(\ast)$ we only use the definition of the dinatural transformation $j$ (Lemma \ref{lem:Htilde-as-canonical-algebra}).
In step $(\dagger)$ we first use the fact that the comodule structure $\delta_{\leftdual{X}} : \leftdual{X} \to H \otimes \leftdual{X}$ of the left dual of an object $(X,\delta_X) \in \lcomod{H}$ is given by
\[ \delta_{\leftdual{X}} = \tau_{\leftdual{X},H} \circ (\id{\leftdual{X}} \otimes S \otimes \ev_{X}) \circ (\id{\leftdual{X}} \delta_X \otimes \id{\leftdual{X}}) \circ (\lcoev_X \otimes \id{X}) . \]
Then we simplify by using the zig-zag identities of duality morphisms.
In step $(\diamond)$ we also use zig-zag identities of duality morphisms.
In step $(\#)$ we use the fact that $f$ is in particular a morphism of left $H$-modules.

By commutativity of the diagram, this computation shows that $\beta_{M,X} = (\mu_H \otimes \id{X\otimes M}) \circ (\id{H}\otimes S\otimes \id{X \otimes M})\circ (\id{H} \otimes \delta_X \otimes \id{M}) \circ (\id{H} \otimes \tau_{M,X})$, which proves our claim.
\end{proof}

\bigskip
To conclude this article we discuss as an application how Theorem \ref{thm:trace-of-bimodule-category-given-by-bicomodule-algebra} together with Theorem \ref{thm:module-category-via-algebra} can be used to recover an equivalence of Yetter-Drinfeld modules and Hopf bimodules, that is closely related to the one in \cite{schauenburg, bespalov+drabant}.

For a finite-dimensional Hopf algebra $H$ over $\kk$ consider the regular $(\lcomod{H})$-bimodule category $\lcomod{H}$.
Viewing it as a left $(\genHopf{H}{H^\opcop}{}{})$-module category we can apply Theorem \ref{thm:module-category-via-algebra} to it.
Taking into account Remark \ref{rem:thm-module-category-via-algebra} the algebra occurring in that theorem can be taken to be $\intHom{\II,\II} \cong \widehat{H} \in \genHopf{H}{H^\opcop}{}{}$.
$\widehat{H}$ is the $\kk$-algebra underlying the Hopf algebra $H$, with $H$-$H^\opcop$-bicomodule structure
\begin{align*}
\delta_{\widehat{H}} : \widehat{H} &\lto H \otimes \widehat{H} \otimes H^\opcop, \\
h &\lmapsto h_{(1)} \otimes h_{(2)} \otimes S(h_{(3)}).
\end{align*}
Theorem \ref{thm:module-category-via-algebra} moreover provides an equivalence of $(\lcomod{H})$-bimodule categories
\[ ? \otimes \widehat{H} : \lcomod{H} \xlongrightarrow{\sim} \genHopf{H}{H^\opcop}{}{\widehat{H}}, \]
which turns out to be given by induction along $\widehat{H}$ (with respect to the right action and coaction), cf. \cite[Equation~(4.4)]{shimizu}.
Since this is a bimodule functor it induces an equivalence between the twisted centers,
\begin{equation} \label{eq:twisted-yetter-drinfeld-hopf-equivalence}
? \otimes \widehat{H} : \twcenter\! \left( \lcomod{H} \right) \xlongrightarrow{\sim} \twcenter\! \left( \genHopf{H}{H^\opcop}{}{\widehat{H}} \right) .
\end{equation}
By Theorem \ref{thm:trace-of-bimodule-category-given-by-bicomodule-algebra} the twisted center of $\genHopf{H}{H^\opcop}{}{\widehat{H}}$ is equivalent to the category $\genHopf{H}{H^\opcop}{\tild{H}}{\widehat{H}}$ of Hopf bimodules.
On the other hand, by Proposition \ref{prop:twisted-center-via-twisted-central-monad} the twisted center of $\lcomod{H}$ is equivalent to $\lmodin{\left( \lcomod{H} \right)}{( \tild{H}\ogreaterthan ?)}$, the category of modules over the monad $(\tild{H} \ogreaterthan ?)$ in $\lcomod{H}$.
An object $X$ in this category is a left $H$-comodule $X$ which is also a left module over the algebra $H$ satisfying the additional compatibility relation
\[ S^2 \left( h_{(1)} \cdot x_{(-1)} \right) \otimes h_{(2)}.x_{(0)} = S^2 \left( (h_{(1)}.x)_{(-1)} \right) \cdot h_{(2)} \otimes (h_{(1)}.x)_{(0)}   \qquad \text{for all } h \in H, x \in X. \]
This is a twisted variant of the Yetter-Drinfeld condition and thus we see that the equivalence \eqref{eq:twisted-yetter-drinfeld-hopf-equivalence} is a twisted variant of the equivalence of Yetter-Drinfeld modules and Hopf bimodules that was shown in \cite{schauenburg,bespalov+drabant}.
We should note, however, that in the references the equivalence is braided monoidal, whereas here we only obtain the (linear) functor itself.

\end{document}